\documentclass[11pt,a4paper]{article}
\usepackage[utf8]{inputenc}
\usepackage{amsmath, amssymb, amsthm}
\usepackage{graphicx,hyperref,geometry}
\usepackage{a4wide}
\usepackage{xcolor,verbatim}

\newtheorem{theorem}[subsection]{Theorem}%[section]
\newtheorem{corollary}[subsection]{Corollary}
\newtheorem{lemma}[subsection]{Lemma}

\newtheorem{remark}[subsection]{Remark}

\newcommand{\cA}{\mathcal{A}}
\newcommand{\cB}{\mathcal{B}}
\newcommand{\cO}{\mathcal{O}}

\title{General sharp upper bounds on the total coalition number}

\author{János Barát\thanks{Research supported by ERC Advanced Grant ”GeoScape” and the National Research, Development and Innovation Office, grant K-131529.} \\
\small Alfr\'ed R\'enyi Institute of Mathematics\\
\small University of Pannonia, Department of Mathematics\\
\small 8200 Veszprém, Egyetem utca 10., Hungary\\
\small \url{barat@renyi.hu} \\
and\\
Zoltán L. Blázsik\thanks{\protect\includegraphics[height=1cm]{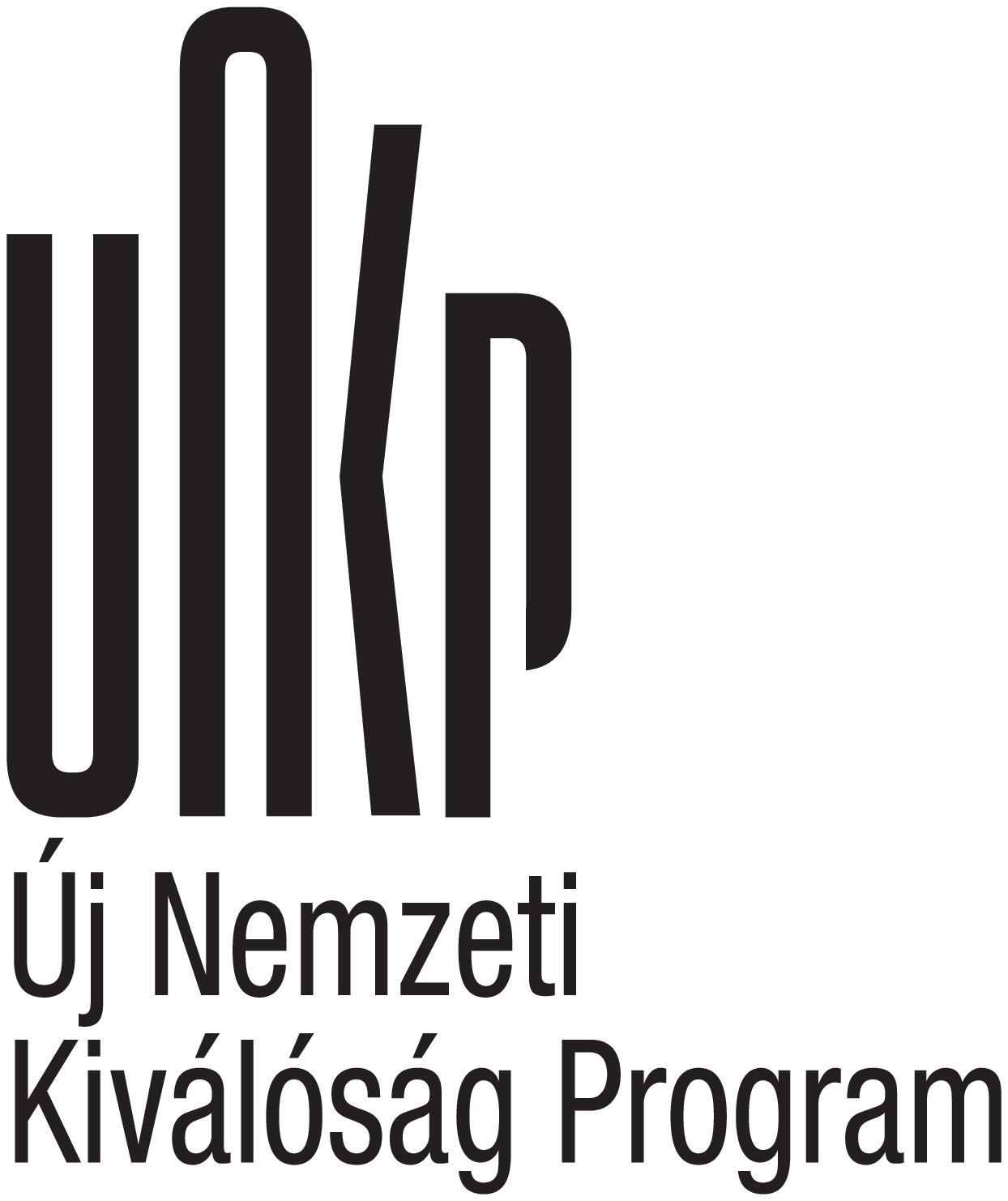}\includegraphics[height=0.8cm]{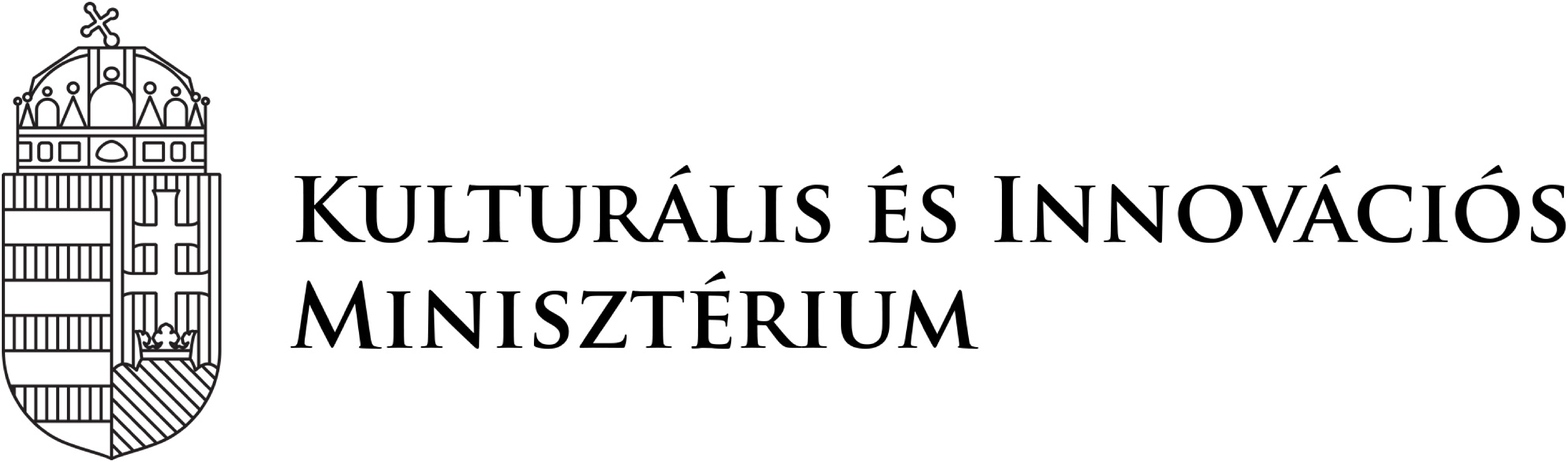} The author was supported by the \'UNKP-22-4-SZTE-480 New National Excellence Program of the Ministry for Culture and Innovation from the source of the National Research, Development and Innovation Fund. The research was supported by the Hungarian National Research, Development and Innovation Office, OTKA grant no. SNN 132625.} \\
\small Alfr\'ed R\'enyi Institute of Mathematics\\
\small MTA--ELTE Geometric and Algebraic Combinatorics Research Group \\
\small SZTE Bolyai Institute \\
\small \url{blazsik@renyi.hu} \\
}
\date{\today}

\begin{document}
\maketitle
\begin{abstract}
Let $G(V,E)$ be a finite, simple, isolate-free graph. Two disjoint sets $A,B\subset V$ form a total coalition in $G$, if none of them is a total dominating set, but their union $A\cup B$ is a total dominating set. A vertex partition $\Psi=\{C_1,C_2,\dots,C_k\}$ is a total coalition partition, if none of the partition classes is a total dominating set, meanwhile for every $i\in\{1,2,\dots,k\}$ there exists a distinct $j\in\{1,2,\dots,k\}$ such that $C_i$ and $C_j$ form a total coalition. The maximum cardinality of a total coalition partition of $G$ is  the total coalition number of $G$ and denoted by $TC(G)$. We give a general sharp upper bound on the total coalition number as a function of the maximum degree.
We further investigate this optimal case and study the total coalition graph.
We show that every graph can be realised as a total coalition graph.
\end{abstract}

\section{Introduction}
There are several problems in combinatorics, which can be formulated as a certain type of domination problem on an appropriate graph. The various domination concepts are well-studied now. However, new concepts are introduced frequently and the interest is growing rapidly. We recommend three fundamental books \cite{HHS1,HHS2,HY} and some surveys \cite{Henning,GH} about domination in general. In this paper, we investigate a new notion, the \emph{total coalition partition} of graphs that is introduced very recently by Alikhani, Bakhshesh and Golmohammadi \cite{TCG} motivated by the similarly defined \emph{coalition partitions}.

\bigskip

Let $G(V,E)$ denote a graph with vertex set $V$ and edge set $E$. The graphs in this paper are finite (i.e. $|V|$ is finite) and without loops or multiple edges.
In other words, the edges correspond to pairs of different vertices, and there can be at most one edge between two vertices. We use the words points and vertices interchangeably.
The \emph{neighborhood} $N(v)$ of a vertex $v$ is the set of vertices adjacent to $v$, i.e. $N(v)=\{u\ |\ uv\in E\}$. The vertices in $N(v)$ are the \emph{neighbors} of $v$. The \emph{degree} $d(v)$ of $v$ is defined as the number of neighbors of $v$.
The vertex $v$ is \emph{isolated} if $d(v)=0$, and \emph{full} if every other vertex is a neighbor of $v$, i.e. $d(v)=|V|-1$. A graph is \emph{isolate-free} if there are no isolated vertices in it. The minimum degree and the maximum degree of $G$ is denoted by $\delta(G)$ and $\Delta(G)$, respectively. A \emph{vertex partition} is a partition of $V$ into pairwise disjoint subsets.
A set of edges $F\subseteq E$ is \emph{independent} (or a \emph{matching}) if the endpoints of the edges are pairwise different. The \emph{maximum matching} of $G$ refers to a matching with the largest possible cardinality, $\nu(G)$. A \emph{vertex cover} is a set of vertices that includes at least one endpoint of every edge of the graph.

A set of vertices $S\subseteq V$ is a \emph{dominating set} if every vertex of $V \setminus S$ is adjacent to at least one vertex of $S$. Similarly, a set of vertices $T\subseteq V$ is a \emph{total dominating set} if every vertex of $V$ is adjacent to at least one vertex of $T$. Usually the interest centers around finding the (total) dominating set with the minimum cardinality, and this notion is called the \emph{(total) domination number}. Another type of problem with wide literature (e.g. \cite{CDH, CH, GH2}) is packing disjoint (total) dominating sets within the same graph. Here the \emph{(total) domatic number} is defined as the most number of (total) dominating sets, which can partition the vertex set of the graph.

In 2020, Haynes et al. \cite{HHHMM} introduced the \emph{coalitions} in graphs. Let $A\subset V$ and $B\subset V$ denote two (disjoint) subsets of $V$. They form a \emph{coalition} if none of them are dominating sets, but their union $A\cup B$ is.
A \emph{coalition partition} is a vertex partition $\Psi=\{C_1,C_2,\dots,C_k\}$ into $k$ non-empty partition classes such that for every $i\in\{1,2,\dots,k\}$ the class $C_i$ is either a dominating set and $|C_i|=1$, or there exists another class $C_j$ so that they form a coalition. The maximum cardinality of a coalition partition is called the \emph{coalition number} of the graph, and denoted by $C(G)$. The \emph{coalition graph}, denoted by $CG(G,\Psi)$, is created by associating the partition classes of a coalition partition $\Psi$ with the vertex set, and the edges correspond to those pair of classes, which form a coalition. In \cite{HHHMM2}, Haynes et al. proved some upper bounds on the coalition number of graphs in terms of $\delta(G)$, and $\Delta(G)$.

In \cite{TCG}, Alikhani, Bakhshesh and Golmohammadi introduced a new notion, the \emph{total coalition partition} of $G$, motivated by the coalition partitions. Similarly, let $A\subset V$ and $B\subset V$ denote two (disjoint) subsets of $V$. They form a \emph{total coalition} if none of them are total dominating sets, but their union $A\cup B$ is. A \emph{total coalition partition} is a vertex partition $\Psi=\{C_1,C_2,\dots,C_k\}$ into $k$ non-empty partition classes such that for every $i\in\{1,2,\dots,k\}$ there exists a distinct $j\in\{1,2,\dots,k\}$ such that $C_i$ and $C_j$ form a total coalition. The maximum cardinality of a total coalition partition is called the \emph{total coalition number} of the graph, and denoted by $TC(G)$. The \emph{total coalition graph}, denoted by $TCG(G,\Psi)$, is created by associating the partition classes of a total coalition partition $\Psi$ with the vertex set, and the edges correspond to those pair of classes, which form a total coalition. Let us call a total coalition graph \emph{optimal} if it corresponds to a total coalition partition with maximum cardinality. In \cite{TCG}, Alikhani, Bakhshesh and Golmohammadi established some bounds on the total coalition number, and studied graphs with small minimum degrees.

This paper is organized as follows. In Section \ref{s:basic}, we collect some basic properties of the total coalition graphs which we use later on, and show that any graph can be realised as a total coalition graph. Our main result is a general sharp upper bound on the total coalition number for arbitrary isolate-free graphs in Section \ref{s:main}. In Section \ref{s:minmax}, we prove another general tight upper bound in terms of both the minimum and maximum degree. In Section \ref{s:tcg}, we determine the total coalition graph if the corresponding partition is of maximum cardinality and the maximum degree is at least 5. Moreover, we show an example with smaller maximum degree such that there exist two optimal total coalition partitions for which the corresponding total coalition graphs are non-isomorphic. We conclude the paper by investigating the possible non-isomorphic optimal total coalition graphs if $\Delta(G)\le 4$.

\section{Total coalition graph and its properties}\label{s:basic}

One can observe a subtle difference between coalitions and total coalitions.
Namely, the graph has to be isolate-free in order to admit a total coalition.
Let $G$ denote an isolate-free graph on $n$ vertices with total coalition number $TC(G)=k$ and maximium degree $\Delta(G)=\Delta$. Let $\Psi=\{C_1,C_2,\dots,C_k\}$ denote a total coalition partition of $G$, and $TCG(G,\Psi)$ the corresponding total coalition graph.

\begin{lemma}[Theorem 2.10 in \cite{TCG}] \label{l:maxfok}
The maximum degree of $TCG(G,\Psi)$ cannot be greater than the maximum degree of $G$, i.e. $\Delta(TCG(G,\Psi))\le \Delta$.
\end{lemma}

\begin{proof}
    For any partition class $C_i$ there exists a vertex $v_i$, which is not dominated by $C_i$. Thus every $C_j$ that forms a total coalition with $C_i$ must contain at least one vertex from the neighborhood of $v_i$ in $G$. Since $v_i$ has at most $\Delta$ neighbors, the degree of $C_i$ in $TCG(G,\Psi)$ is also at most $\Delta$. Hence $\Delta(TCG(G,\Psi))\le \Delta$ follows.
\end{proof}

One can establish a connection between the size of the maximum matching of $TCG(G,\Psi)$ and $\Delta$ in the following way. 

\begin{lemma} \label{l:matching}
    The size of the maximum matching is at most $\Delta$, i.e. $\nu(TCG(G,\Psi))\le\Delta$.
\end{lemma}

\begin{proof}
    It is well-known that any total dominating set has cardinality  at least $\frac{n}{\Delta}$. Therefore if we consider an edge $C_iC_j$ of $TCG(G,\Psi)$, then $|C_i|+|C_j|=|C_i \cup C_j|\ge\frac{n}{\Delta}$ holds. Since $\displaystyle \left|\bigcup_{i=1}^k C_i \right|=\sum_{i=1}^k |C_i|=n$, the size of the maximum matching has to be at most $\Delta$.
\end{proof}

There is an immediate corollary regarding the case of equality.

\begin{corollary} \label{c:equality}
If $\nu(TCG(G,\Psi))=\Delta$, then for every edge $f$ of the maximum matching, the size of the union of the partition classes corresponding to the endpoints of $f$ is exactly $\frac{n}{\Delta}$. Thus, there can be no other vertices in the total coalition graph, since these partition classes consume all the vertices. Hence the number of partition classes in this case is exactly $2\Delta$.
\end{corollary}

The next observation helps to understand the maximum degree vertices of the total coalition graph if $\Delta(TCG(G,\Psi))=\Delta$.

\begin{lemma}\label{l:cover}
    If $\Delta(TCG(G,\Psi))=\Delta$, then the neighbors of any vertex with maximum degree in $TCG(G,\Psi)$ form a vertex cover. %in $TCG(G,\Psi)$.
\end{lemma}

\begin{proof}
    Suppose $C_i$ is a vertex of $TCG(G,\Psi)$ with degree $\Delta$. Since $C_i$ is not a total dominating set in $G$, there exists a vertex $v_i$ in $G$, which is not dominated by $C_i$. This $v_i$ must be dominated by those partition classes which are adjacent to $C_i$ in the total coalition graph. There are $\Delta$ such partition classes, hence $v_i$ also has degree $\Delta$ in $G$. All its neighbors must correspond to those $\Delta$ partition classes, which are the neighbors of $C_i$ in the total coalition graph.

    Observe that $v_i$ can be dominated only by these $\Delta$ partition classes, hence the corresponding vertices in the total coalition graph form a vertex cover.
\end{proof}

Haynes et al. \cite{HHHMM3} proved that any graph $G$ can be realised as a coalition graph.
Based on their construction, we show the analogous statement for total coalition graphs.

\begin{lemma}
    For any isolate-free graph $G$, there exists another graph $H$ with a total coalition partition $\pi$ of $H$ such that $G$ is isomorphic to $TCG(H,\pi)$.
\end{lemma}

\begin{proof}
Let the vertices of $G$ be $\{v_1,\dots,v_n\}$ and the edges be $\{e_1,\dots,e_m\}$.
We define three type of vertices of $H$.
One corresponds to the vertices of $G$, one to the edges of $G$ and one to the non-edges of $G$.
With a slight abuse of notation, let the first type of vertices be $\{v_1,\dots,v_n\}$. 
They span a complete graph on $n$ vertices in $H$.
In the partition $\pi$, we put $v_i$ into class $V_i$ for every $1\le i\le n$.
Secondly, for every edge $e_i=v_jv_k$ of $G$, we define two vertices $u_{ij}$ and $u_{ik}$ of $H$ such that
$u_{ij}$ is adjacent to $\{v_1,\dots,v_n\}$ except $v_j$, vertex $u_{ik}$ is adjacent to $\{v_1,\dots,v_n\}$ except $v_k$.
We put $u_{ij}$ to $V_k$ and $u_{ik}$ to $V_j$.
Thirdly, for every non-edge $v_jv_k$ of $G$, we define a vertex $x_{jk}$ of $H$ such that $x_{jk}$ is adjacent to $\{v_1,\dots,v_n\}$ except $v_j$ and $v_k$.
We put $x_{jk}$ to any $V_i$, where $i\neq j,k$.

First notice that any partition class $V_i$ is not a total dominating set.
Indeed, if $v_iv_j$ is a non-edge of $G$, then $V_i$ does not dominate $x_{ij}$.
If $e_k=v_iv_j$ is an edge of $G$, then $u_{ki}$ is not adjacent to $v_i$ and neither is any other vertex in $V_i$, since they are from the second or third class.

Second, we have to show that $V_j\cup V_k$ is a total dominating set if and only if $v_jv_k$ is an edge of $G$.
Assume $v_jv_k$ is an edge of $G$.
Since $v_j\in V_j$ and $v_k\in V_k$ the set $V_j\cup V_k$ dominates every vertex from the first class and the third class.
Although $v_j$ is not adjacent to vertices of form $u_{ij}$, the vertex $v_k$ dominates them. Vertex $v_j$ dominates all other vertices in the second class.

Assume now that $v_jv_k$ is a non-edge of $G$.
By definition $x_{jk}$ is not adjacent to either $v_j$ or $v_k$.
Therefore $V_j\cup V_k$ is not a total dominating set.
\end{proof}

\section{General upper bound on $TC(G)$} \label{s:main}

Alikhani et al. proved some upper bounds on $TC(G)$ if $\delta(G)$ is precisely 1 or 2.

\begin{theorem}[Theorem 3.5. in \cite{TCG}] \label{delta1}
For any graph $G$ with $\delta(G)=1$, $TC(G)\le \Delta(G)+1$.
\end{theorem}

\begin{theorem}[Theorem 4.2. in \cite{TCG}]\label{delta2}
For any graph $G$ with $\delta(G)=2$, $TC(G)\le 2\Delta(G)$.
\end{theorem}

However, they did not provide a general upper bound on $TC(G)$. One might wonder how does the optimal structure of a total coalition partition looks like. Either the size of the partition classes are balanced or there is one large class, that forms a total dominating set with any other class or something in between.
For instance, a few fairly large classes, that form total coalitions with multiple other classes. These options can be phrased in the language of the total coalition graph, as well.
The second option means that in the total coalition graph there is a full vertex. Moreover, there are examples where determining the optimal structure is not possible because the optimum can be reached by different structures, see Figure \ref{fig:opt}.

\begin{figure}[!h]
    \centering
    \includegraphics[width=0.7\textwidth]{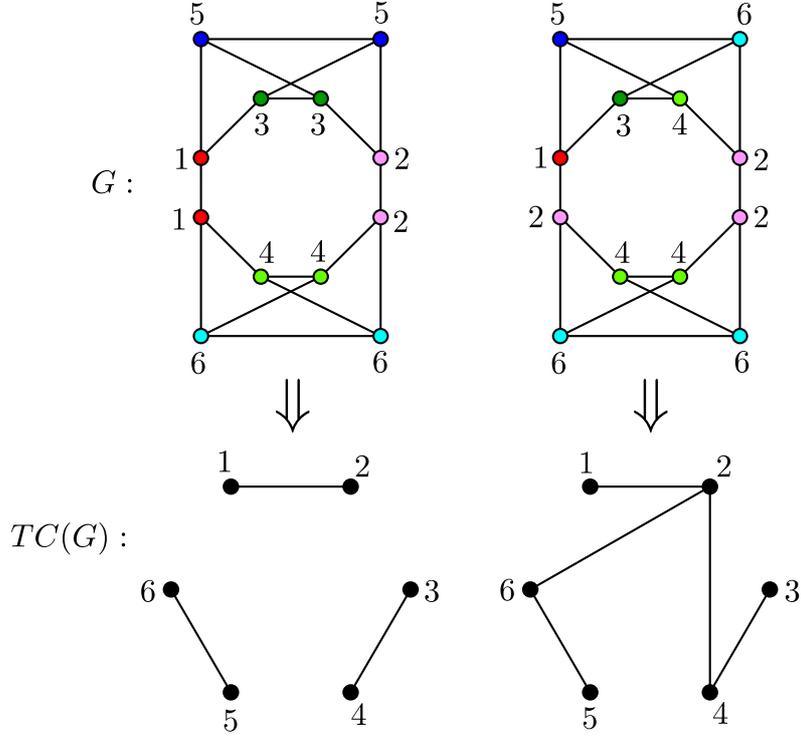}
    \caption{Different partitions reach the optimum for the same graph with non-isomorphic total coalition graphs.}
    \label{fig:opt}
\end{figure}

By Lemma \ref{l:maxfok}, if there is a full vertex in the total coalition graph, then the total coalition graph has at most $\Delta(G)+1$ vertices.
On the other hand, if the total coalition partition is balanced, in other words the classes have almost the same size, then the maximum number of partition classes are bounded from above by $2\Delta(G)$. By Lemma \ref{l:matching}, the union of two partition classes, which form a total coalition, must have size at least $\frac{|V(G)|}{\Delta(G)}$, hence the classes must have size at least $\frac{|V(G)|}{2\Delta(G)}$.

These observations suggest that the upper bounds from \cite{TCG} are reasonable. On the flip side, one might think that combining the two approaches, namely using multiple classes with relatively high degree within the total coalition graph, can lead to even more partition classes. As the next theorem shows, there are graphs for which the total coalition number is quadratic in terms of $\Delta(G)$.

\begin{theorem}\label{t:pelda}
   For any $\Delta\ge 3$, there exists a graph $G$ such that $\Delta(G)=\Delta$ and $$TC(G)\ge \left\{\begin{array}{cc}
        \frac{\Delta^2}{4}+\Delta + \frac{3}{4} & \mathrm{if~} \Delta \mathrm{~is~odd,} \\
        \frac{\Delta^2}{4}+\Delta + 1& \mathrm{if~} \Delta \mathrm{~is~even.}
    \end{array} \right.$$
\end{theorem}

\begin{proof}
The constructions are very similar for different parities.
Thus we elaborate on the even case, and after that we point out the small changes in the details for the odd case. Suppose $\Delta=2r$. We partition the vertex set into $r+1$ large classes $\{C_1,C_2,\dots,C_{r+1}\}$ and the rest such that for every $i\in\{1,2,\dots,r+1\}$ the class $C_i$ is almost a total dominating set.
More precisely, there is exactly one non-dominated vertex $v_i$ with respect to $C_i$. It happens to be the case that $v_i$ also belongs to the class $C_i$, but all the remaining vertices can form singleton partition classes.

Our construction uses 3 types of building blocks illustrated in Figure \ref{fig:negyzetes}. The first one is a complete graph on the vertices $\{v_1,v_2,\dots,v_{r+1}\}$. The second one, denoted by $\cO$, consists of the further neighbors of $\{v_1,v_2,\dots,v_{r+1}\}$ such that there are $r$ vertices $\{C_{i,1},C_{i,2},\dots,C_{i,r}\}$ adjacent to $v_i$ for each $i\in\{1,2,\dots,r+1\}$.
Each vertex of $\cO$ forms a singleton partition class.
Its only purpose is to cover the corresponding $v_i$.

\begin{figure}[!h]
    \centering
    \includegraphics[width=\textwidth]{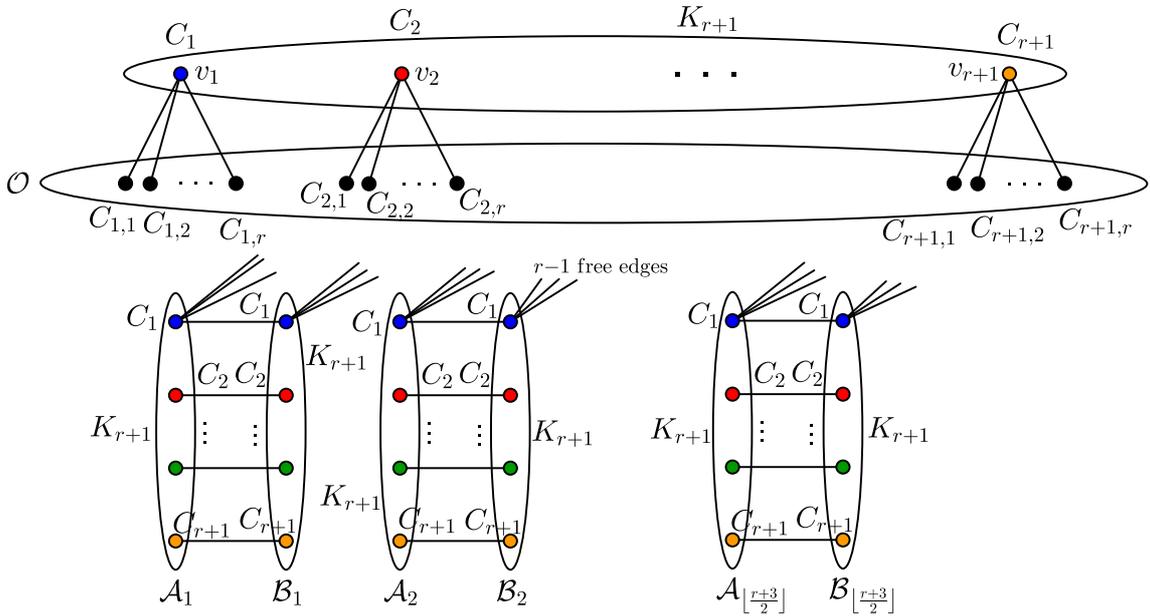}
    \caption{The sketch of a graph $G$ with $TC(G)\ge \frac{\Delta^2}{4}+\Delta+1$ for $\Delta=2r$.}
    \label{fig:negyzetes}
\end{figure}

The third block utilizes the same gadget $\left \lfloor\frac{r+3}{2}\right \rfloor$ times. The gadget consists of two complete graphs $\cA_i$ and $\cB_i$ on $r+1$ vertices each, and a perfect matching between them. All the vertices of the third block belong to $\{C_1,C_2,\dots,C_{r+1}\}$.
In both $\cA_i$ and $\cB_i$ each class of $\{C_1,C_2,\dots,C_{r+1}\}$ is represented exactly once such that the endpoints of the edges of the perfect matchings always belong to the same partition class.

Notice that the vertices of the third block have degree $r+1$ at this point. Therefore there are $r-1$ possible new edges incident to each of them. Since for every $i\in\{1,2,\dots,r+1\}$ the class $C_i$ must dominate all vertices but $v_i$, we use the possible new edges to connect the vertices of the class $C_i$ to those vertices of $\cO$, which are not yet dominated. That is $\cO\setminus\{C_{i,1},C_{i,2},\dots,C_{i,r}\}$ for the class $C_i$. We need $\left \lfloor \frac{r+3}{2} \right \rfloor$ copies of this gadget because $|\cO\setminus\{C_{i,1},C_{i,2},\dots,C_{i,r}\}|=r^2$, and the two points from $C_i$ in the same gadget can dominate at most $2(r-1)$. We stop using the possible new edges from the two vertices of $C_i$ of the last copy as soon as the class $C_i$ dominates all vertices but $v_i$. Observe that after using these additional edges to dominate the vertices of $\cO$ for all the classes of $\{C_1,C_2,\dots,C_{r+1}\}$ the degree of the vertices of $\cO$ is exactly $r+1$.
Hence in this graph $G$, the maximum degree is $\Delta=2r$, and the partition has $r(r+1)+r+1=(r+1)^2=\frac{\Delta^2}{4}+\Delta+1$ classes. It is straightforward to check that this is a total coalition partition. None of the classes forms a total dominating set alone, and for every $i\in\{1,2,\dots,r+1\}$ the class $C_i$ forms a total coalition with $C_{i,j}$ for each $j\in\{1,2,\dots,r\}$.

If $\Delta=2r+1$, then the construction changes very slightly. For every $i\in\{1,2,\dots,r+1\}$ there are 1 more neighbor $C_{i,r+1}$ of $v_i$ in $\cO$, and every vertex of the third block has 1 more possible new edge. Thus in the odd case, we need $\left \lceil \frac{r+1}{2} \right \rceil$ copies of the gadget because for each class $C_i$ we have to dominate $r(r+1)$ points in $\cO$ and the two points from the same gadget can dominate at most $2r$. The same argument with the same large classes works again, and it gives that the number of classes in this total coalition partition is $r+1+(r+1)(r+1)=\frac{\Delta+1}{2}\frac{\Delta+3}{2} = \frac{\Delta^2}{4}+\Delta+\frac{3}{4}$.
\end{proof}

\begin{remark}
    The construction works only for $\Delta\ge3$. For smaller values of $\Delta$, there is no room for possible new edges from the third block. However, there are graphs $G_1$ and $G_2$ for $\Delta=1$ and $2$, respectively, such that the their total coalition number reaches the same bound. For example $G_1=K_2$ and $G_2=C_4$ suffice.
\end{remark}

For $\Delta\ge6$, this construction shows the total coalition number can exceed $2\Delta$, moreover it is quadratic in $\Delta$. In the next theorem, we show that this construction is the best possible by proving the same upper bound on the total coalition number for any isolate-free graph.

\begin{theorem} \label{t:quadratic}
    For any isolate-free graph $G$, $TC(G)\le \left (\frac{\Delta(G)+2}{2}\right )^2=\frac{\Delta(G)^2}{4}+\Delta(G)+1$  holds.
\end{theorem}

\begin{proof}
    Consider a total coalition partition $\Psi=\{C_1,C_2,\dots,C_k\}$ of $G$ such that $TC(G)=k$. Focus on the corresponding total coalition graph $TCG(G,\Psi)$. By Lemma \ref{l:maxfok}, we know that $\Delta(TCG(G,\Psi))\le \Delta(G)$. By Lemma \ref{l:matching} we get that there are at most $\Delta(G)$ independent edges in $TCG(G,\Psi)$. If $\nu(TCG(G,\Psi))=\Delta(G)$, then by Corollary~\ref{c:equality} we immediately get that $TC(G)=2\Delta(G)<\left (\frac{\Delta(G)+2}{2}\right )^2$.

    Suppose $\nu(TCG(G,\Psi))=m<\Delta(G)$ and fix a maximum matching $M$ with $m$ edges. For any edge $C_iC_j$ of $M$, let us estimate the number of additional vertices of the total coalition graph, which are adjacent to either $C_i$ or $C_j$ or both. If there are edges connecting additional vertices to both $C_i$ and $C_j$, then it leads to a contradiction to $M$ being a maximum matching unless there is only one additional vertex which is adjacent to both $C_i$ and $C_j$.

    Otherwise at most one of the partition classes $C_i,C_j$ is adjacent to any additional vertices. Assume  $C_i$ is adjacent to $d$ additional vertices of the total coalition graph, while $C_j$ is not adjacent to any additional vertices. Since $C_i$ is not a total dominating set in $G$, there exists a vertex $v_i$ of $G$, which is not dominated by $C_i$. All of those partition classes, which are adjacent to $C_i$, must contain at least one vertex from the neighborhood of $v_i$ in $G$. The $m-1$ other edges of $M$ also give rise to total coalitions, hence at least one end-vertex of these edges also must contain at least one vertex from the neighborhood of $v_i$ in $G$. This gives a bound on $d$ with respect to $\Delta(G)$ and $\nu(TCG(G,\Psi))$: $$\Delta(G)-(d+1)\ge m-1 ~\quad~\Longleftrightarrow~\quad~\Delta(G)-m\ge d.$$

    Since $\Delta(G)-m\ge 1$, there can be more additional vertices in the second case. This leads to an upper bound $k\le m(\Delta(G)-m+2)$. The right-hand side is a quadratic function of $m$, and it takes its maximum value if $m=\frac{\Delta(G)+2}{2}=\frac{\Delta(G)}{2}+1$. Hence $TC(G)\le \left (\frac{\Delta(G)+2}{2}\right )^2=\frac{\Delta(G)^2}{4}+\Delta(G)+1$.
\end{proof}

\begin{remark}\label{r:qodd}
    If $\Delta(G)$ is odd, then $\frac{\Delta(G)+2}{2}$ is not an integer. Therefore, the maximum value of the quadratic function is taken by choosing $m_1=\left \lfloor \frac{\Delta(G)+2}{2} \right \rfloor = \frac{\Delta(G)+1}{2}$ or $m_2=\left \lceil \frac{\Delta(G)+2}{2} \right \rceil = \frac{\Delta(G)+3}{2}$. In both cases, we get a slightly improved upper bound $TC(G)\le \frac{\Delta(G)+1}{2}\frac{\Delta(G)+3}{2}=\frac{\Delta(G)^2}{4}+\Delta(G)+\frac{3}{4}$.
\end{remark}

Theorem \ref{t:pelda} and \ref{t:quadratic} shows that our general upper bound for isolate-free graphs is sharp for any $\Delta\ge1$.

\section{Upper bound on $TC(G)$ in terms of $\delta(G)$ and $\Delta(G)$} \label{s:minmax}

Motivated by \cite{TCG}, where they studied the cases $\delta(G)=1,2$, we are intrigued to find a general upper bound, which depends not only on the maximum degree, but on the minimum degree, too. By improving Lemma \ref{l:matching}, we are able to deduce a general upper bound on the total coalition number in terms of $\delta(G)$ and $\Delta(G)$.

As before, let $G$ denote an isolate-free graph with $TC(G)=k$ and let $\Psi=\{C_1,C_2,\dots,C_k\}$ denote a total coalition partition of $G$, and $TCG(G,\Psi)$ the corresponding total coalition graph.

\begin{lemma}[improved version of Lemma \ref{l:matching}] \label{l:imp}
    The size of the maximum matching in $TCG(G,\Psi)$ is at most the minimum degree of $G$, i.e. $\nu(TCG(G,\Psi))\le \delta(G)$.
\end{lemma}

\begin{proof}
    Consider a vertex $v\in V$ with $d(v)=\delta(G)$. The pair of partition classes corresponding to the endpoints of any edge of the maximum matching forms a total coalition. Thus these pairwise disjoint total dominating sets must dominate $v$ as well. However, $v$ can be dominated only via its neighbors and $|N(v)|=\delta(G)$ hence $\nu(TCG(G,\Psi))\le \delta(G)$ holds.
\end{proof}

Now, let us incorporate this improved bound into the proof of Theorem \ref{t:quadratic}. Recall that the size of the maximum matching in the total coalition graph of the graph attaining the upper bound of Theorem \ref{t:quadratic} is $\left\lfloor \frac{\Delta(G)+2}{2} \right \rfloor$ thus the improved bound on $\nu(TCG(G,\Psi))$ is a restriction only if $\delta(G)<\left \lfloor\frac{\Delta(G)+2}{2}\right \rfloor$.

\begin{theorem}\label{t:minmax}
    If $G$ is an isolate-free graph with $\delta(G)<\left \lfloor\frac{\Delta(G)+2}{2}\right \rfloor$, then $TC(G)\le \delta(G)(\Delta(G)-\delta(G)+2)$ holds.
\end{theorem}

\begin{proof}
    The argument is the same as in the proof of Theorem \ref{t:quadratic}, but now the quadratic upper bound $f(m)=m(\Delta(G)-m+2)$ where $m=\nu(TCG(G,\Psi))$ on the number of partition classes in the total coalition partition cannot take its maximum value because we assumed that $\delta(G)<\left \lfloor\frac{\Delta(G)+2}{2}\right \rfloor$ and parameter $m=\nu(TCG(G,\Psi))\le \delta(G)$ by Lemma \ref{l:imp}.

    Nevertheless, the leading coefficient of $f(m)$ is negative thus the graph of the quadratic function is a downward open parabola. Since $0$ and $\frac{\Delta(G)+2}{2}$ are the two roots of this function hence it takes its maximum value with respect to the constraint $\delta(G)<\left \lfloor\frac{\Delta(G)+2}{2}\right \rfloor$ by choosing $m=\nu(TCG(G,\Psi))=\delta(G)$. Thus $TC(G)\le f(\delta(G))=\delta(G)(\Delta(G)-\delta(G)+2)$.
\end{proof}

For $\delta(G)=1,2$ this bound gives back the results of \cite{TCG}. The following theorem proves that this bound is also tight, and the construction attaining the bound is very similar to the one in the proof of Theorem \ref{t:pelda}.

\begin{theorem}\label{t:peldaminmax}
   For any $\Delta\ge 2$, there exists an isolate-free graph $G$ such that $\delta(G)<\left \lfloor\frac{\Delta(G)+2}{2}\right \rfloor$ and $TC(G)\ge\delta(G)(\Delta(G)-\delta(G)+2)$.
\end{theorem}

\begin{proof}
Using the same argument as before, it is straightforward to check that the partition illustrated in Figure \ref{fig:minmax} is indeed a total coalition partition with $\delta(G)(\Delta(G)-\delta(G)+2)$ partition classes.

\begin{figure}[!h]
    \centering
    \includegraphics[width=\textwidth]{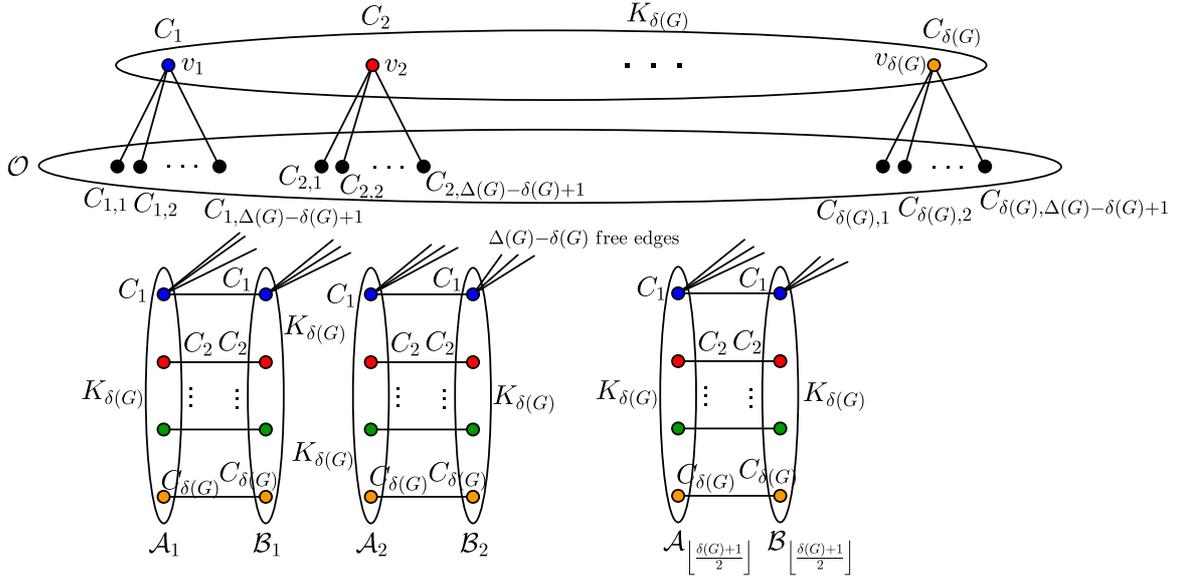}
    \caption{The sketch of a graph $G$ with $TC(G)\ge \delta(G)(\Delta(G)-\delta(G)+2)$, if $\delta(G)<\left \lfloor\frac{\Delta(G)+2}{2}\right \rfloor$.}
    \label{fig:minmax}
\end{figure}

Notice that each vertex in the third block has $\delta(G)$ neighbors within the third block, therefore they have $\Delta(G)-\delta(G)$ possible new edges to dominate the not yet dominated vertices of $\cO$. For every $i\in\{1,2,\dots,\delta(G)\}$ the class $C_i$ needs to dominate $\cO\setminus\{C_{i,1},C_{i,2},\dots,C_{i,\Delta(G)-\delta(G)+1}\}$ which is $(\delta(G)-1)(\Delta(G)-\delta(G)+1)$ vertices of $\cO$ and for each copy of the gadget the two vertices of $C_i$ have $2(\Delta(G)-\delta(G))$ possible new edges. Hence $\left \lfloor \frac{\delta(G)+1}{2} \right \rfloor$ copy suffices since $\Delta(G)-\delta(G)\ge\delta(G)-1$ if $\delta(G)<\left \lfloor\frac{\Delta(G)+2}{2}\right \rfloor$ and $\Delta(G)\ge2$.

However, we need to be careful: the minimum degree of our construction must be exactly $\delta(G)$. The vertices in $\cO$ has the least number of neighbors, one from the first block and $\delta(G)-1$ from the third block which is exactly $\delta(G)$. Thus we stop using the possible new edges from the last copy of the gadget as soon as the partition class $C_i$ dominates all vertices but $v_i$.
\end{proof}

\begin{remark}
    If $\Delta(G)=1$ then consequently $\left \lfloor\frac{\Delta(G)+2}{2}\right \rfloor = 1 > \delta(G)$ leads to a contradiction because the minimum degree of any isolate-free graph has to be at least 1.
\end{remark}

\section{What does the total coalition graph look like?}\label{s:tcg}

From the proof of Theorem \ref{t:quadratic}, we can recognize a specific subgraph of the total coalition graph if the corresponding total coalition partition is of maximum cardinality. But what can we say about the whole total coalition graph? What can it consist of beyond the union of $\left \lfloor \frac{\Delta(G)+2}{2} \right \rfloor$ stars?

By Remark \ref{r:qodd}, we can immediately see that if $\Delta(G)$ is odd, then the total coalition graph can have $\left \lfloor \frac{\Delta(G)+2}{2}\right \rfloor$ stars on $\left \lceil \frac{\Delta(G)+2}{2}\right \rceil$ vertices or the other way around, $\left \lceil \frac{\Delta(G)+2}{2}\right \rceil$ stars on $\left \lfloor \frac{\Delta(G)+2}{2}\right \rfloor$ vertices. In the following, we prove that if $\Delta(G)\ge 5$, then this can be the only difference.

\begin{theorem}\label{t:stareven}
    If $G$ is an isolate-free graph with $\Delta(G)$ even and $\Delta(G)\ge 6$ such that $TC(G)=\frac{\Delta(G)^2}{4} + \Delta(G)+1$, then the total coalition graph is determined up to isomorphism, it does not depend on the realising total coalition partition.
\end{theorem}

\begin{proof}
    Since $TC(G)=\frac{\Delta(G)^2}{4} + \Delta(G)+1$, the total coalition graph must contain the disjoint union of $\frac{\Delta(G)+2}{2}$ stars on $\frac{\Delta(G)+2}{2}$ vertices and the size of the maximum matching is also equal to $\frac{\Delta(G)+2}{2}$ by the proof of Theorem \ref{t:quadratic}. It is clear that these are the only vertices of the total coalition graph, and the question is what type of further edges can occur?

\begin{figure}[!ht]
    \centering
    \includegraphics[width=\textwidth]{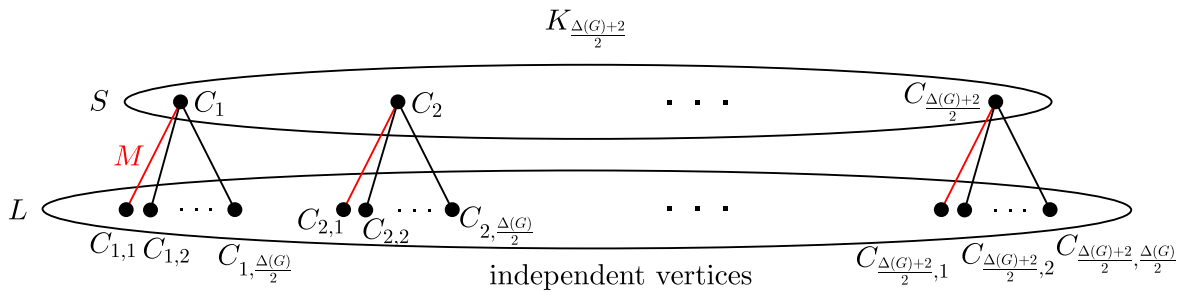}
    \caption{The total coalition graph if $\Delta(G)\ge 6$, $\Delta(G)$ even and $TC(G)=\frac{\Delta(G)^2}{4} + \Delta(G)+1$.}
    \label{fig:tcg}
\end{figure}

    Denote the set of centers of these stars by $S$ and the leaves by $L$. These stars have at least 3 leaves since $\Delta(G)\ge 6$.
    Hence an edge between any two vertices of $L$ yields a larger matching, a contradiction. %itt kell a legalább 3 levél
    Next, there can be no edges between $S$ and $L$. Suppose to the contrary that there is an edge $s\ell$ between $s\in S$ and $\ell\in L$. Let us fix a maximum matching $M$ avoiding $\ell$. This can be done because $M$ consists of $\frac{\Delta(G)+2}{2}$ independent edges incident to each vertex of $S$, and in every star there are at least $3$ possibilities to choose from, thus $\ell$ can be avoided. %már ha két különböző levél lenne, akkor is igaz lenne.

    Consider the set of vertices $C_s$ of $G$ corresponding to $s$ in the total coalition partition and fix a vertex $v$, which is not dominated by $C_s$. Only the vertices of $N(v)$ can dominate $v$, so we can count the number of partition classes, which have to intersect $N(v)$ non-trivially. There are $\frac{\Delta(G)+2}{2}-1$ leaf neighbors of $s$ in the corresponding star. There are $\frac{\Delta(G)+2}{2}-1$ edges of $M$ avoiding $s$. For each such matching edge, the two partition classes corresponding to the endpoints form a total coalition. Therefore at least one of these partition classes must intersect $N(v)$. Moreover, the partition class $C_{\ell}$ corresponding to the leaf $\ell$ also has to dominate $v$. Altogether this is $\left (\frac{\Delta(G)+2}{2}-1\right )+\left (\frac{\Delta(G)+2}{2}-1 \right )+1 = \Delta(G)+1$ partition classes and all of them must intersect $N(v)$ non-trivially, which contradicts with $|N(v)|\le \Delta(G)$. Hence the further edges of the total coalition graph must be spanned by $S$.

    We claim that the subgraph spanned by $S$ must be a complete graph. Suppose on the contrary that for some $s\ne s'\in S$ the edge $ss'$ is missing from the total coalition graph. Thus the partition class $C_{s'}$ corresponding to $s'$ does not intersect $N(v)$, but from the argument of the previous paragraph we know that no matter how we chose the edge for $M$ from the star centered around $s'$ the union of the two partition classes corresponding to the two endpoints must intersect $N(v)$. Hence all the partition classes corresponding to the leaves of the star with center $s'$ must intersect $N(v)$ which contradicts again to  $|N(v)|\le \Delta(G)$. %itt is az kell, hogy a csillagnak legyen legalább két ága
    Hence the total coalition graph is determined up to isomorphism as illustrated in Figure~\ref{fig:tcg}.
\end{proof}

Consider an isolate-free graph attaining the upper bound of Theorem \ref{t:quadratic}. By the proof of Theorem \ref{t:quadratic}, the size of the maximum matching is fixed, thus the existence of the three types of further edges depends only on the number of leaves in each star. There can be no further edges spanned by $L$, if there are at least three leaves in each star. For the other two types of edges, between $S$ and $L$, and spanned by $S$, it is necessary to have at least two leaves to conclude that there are no edges between $S$ and $L$ and that the subgraph spanned by $S$ is complete. With these observations we are able to show a similar result if the maximum degree is odd.

\begin{theorem}\label{t:starodd}
    If $G$ is an isolate-free graph with $\Delta(G)\ge 5$, $\Delta(G)$ odd such that $TC(G)=\left \lfloor \frac{\Delta(G)^2}{4} + \Delta(G)+1 \right \rfloor$, then the total coalition graph $H$ is determined up to isomorphism if the size of the maximum matching is fixed.
\end{theorem}

\begin{proof}
    By the proof of Theorem \ref{t:quadratic}, we know that $\nu(H)$ is equal to either $\left \lfloor \frac{\Delta(G)+2}{2} \right \rfloor$ or $\left \lceil \frac{\Delta(G)+2}{2} \right \rceil$. Once the size of the maximum matching in the total coalition graph is fixed then by the observations above the existence of the further edges depends again only on the number of leaves in each star. Since $\Delta(G)\ge 5$, the number of leaves in both cases are at least 3 thus the only further edges are spanned by $S$ and the subgraph spanned by $S$ is complete.
\end{proof}

It remains to consider the case $\Delta(G)\le 4$. Here the optimal total coalition graph is not necessarily unique even if the size of the maximum matching is fixed. For instance, the total coalition graphs corresponding to the two total coalition partitions illustrated in Figure~\ref{fig:opt} are not isomorphic, although both are optimal and have three independent edges.

\section{Possible optimal total coalition graphs if $\Delta(G)\le 4$}

Here we investigate the remaining cases if $TC(G)=\left \lfloor \frac{\Delta(G)^2}{4} + \Delta(G) + 1 \right \rfloor$ and $\Delta(G)\le4$. If $\Delta(G)=1$, then an optimal total coalition partition has 2 classes and they form a total coalition, hence the corresponding total coalition graph $K_2$ is uniquely determined.

Let us assume $\Delta(G)=2$.
Since $TC(G)=4$, there are two independent edges in the total coalition graph by Theorem~\ref{t:quadratic}.
Since the maximum degree of the total coalition graph is also at most 2 by Lemma \ref{l:maxfok}, there are only three candidates: $C_4$, $\overline{C_4}$, $P_4$ for the total coalition graph.
Figure~\ref{fig:delta2} shows that all of these possible total coalition graphs are feasible. However, it is straightforward that only $C_4$ is possible if $G$ is connected.

\begin{figure}[!ht]
    \centering
    \includegraphics[width=0.7\textwidth]{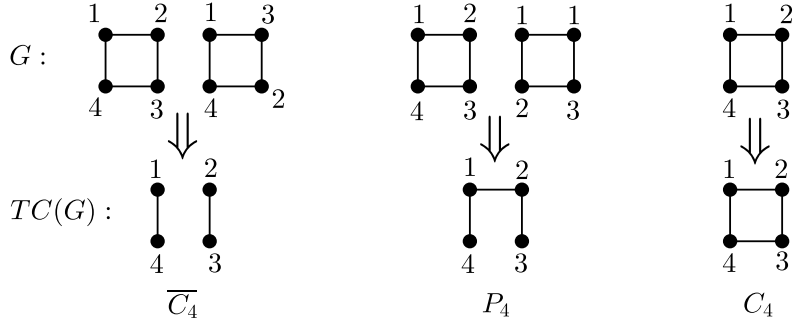}
    \caption{The three possible total coalition graphs if $\Delta(G)=2$ and $TC(G)=4$.}
    \label{fig:delta2}
\end{figure}

Suppose $\Delta(G)=4$. By the proof of Theorem~\ref{t:quadratic}, we know the total coalition graph contains three stars and each of them has 2 leaves. Since these stars have 2 leaves, the observation above yields that $S$ spans a triangle and there are no edges between $S$ and $L$. The subgraph spanned by $L$ must be empty because the centers of the stars have maximum degree by Lemma~\ref{l:maxfok} and their neighbors form a vertex cover by Lemma~\ref{l:cover} and $|S|=3$. Thus the total coalition graph is again determined up to isomoprhism.

The most interesting and complicated case is $\Delta(G)=3$. The size of the maximum matching of the total coalition graph is either 2 or 3 by Remark \ref{r:qodd}. If the size of the maximum matching is 2, then the stars have 2 leaves. Hence the two centers are adjacent, and there are no edges between $S$ and $L$. Similarly to the previous case an edge between the two leaves of the same star contradicts Lemma~\ref{l:cover}. There can be no edges between leaves from different stars either, since that would increase the size of the maximum matching. Therefore, if the size of the maximum matching is 2, then the total coalition graph is uniquely determined.

Lastly, suppose the size of the maximum matching is 3 and $\Delta(G)=3$. We have already seen two possibilities from the two examples in Figure~\ref{fig:opt}. Observe that $K_{3,3}$ is also a feasible total coalition graph if $G=K_{3,3}$ as well and every vertex forms a singleton partition class.\footnote{this property holds for any complete multipartite graph}

By Lemma~\ref{l:maxfok}, the maximum degree of the total coalition graph is at most 3. If the total coalition graph has a vertex $v$ of degree 3, then for any maximum matching $M$ the three neighbors of $v$ form a vertex cover by Lemma~\ref{l:cover}. Hence the three neighbors of $v$ intersect each edge of $M$ exactly once. These restrictions do not determine the total coalition graph, but significantly narrow down the candidates. There are 13 non-isomorphic possible optimal total coalition graphs remaining, as illustrated in Figure~\ref{fig:maradek}.

\begin{figure}[!ht]
    \centering
    \includegraphics[width=\textwidth]{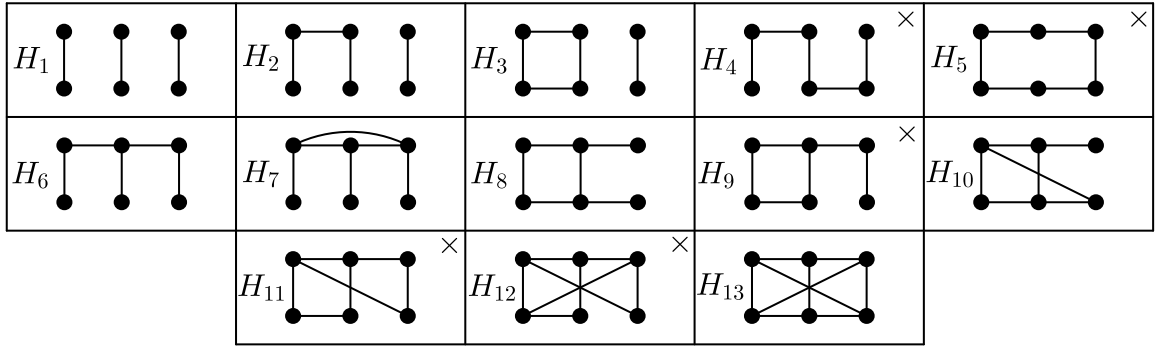}
    \caption{The optimal total coalition graph candidates for $\Delta(G)=3$ and $\nu(H)=3$.}
    \label{fig:maradek}
\end{figure}

We have already realised 3 of them ($H_1$,$H_6$, and $H_{13}\simeq K_{3,3}$).
We show that most of them are realisable with the exception of the 5 graphs marked by $\times$.

\begin{lemma}\label{lem:6kor}
    If an optimal total coalition graph $H$ contains a path $V_1 V_2 V_3 V_4 V_5 V_6$ on 6 vertices, then the edge $V_2 V_5$ must be present in $H$, too. If $V_1 V_2 V_3 V_4 V_5 V_6$ form a cycle of length $6$, then $H\simeq H_{13}\simeq K_{3,3}$.
\end{lemma}

\begin{proof}
    Consider the path $V_1 V_2 V_3 V_4 V_5 V_6$ on 6 vertices.
    Suppose to the contrary $V_2 \cup V_5$ is not a total dominating set. Thus there exists a vertex $v$ that is not dominated by $V_2 \cup V_5$. Since $V_1 V_2$ and $V_2 V_3$ are edges of $H$, therefore $V_1 \cap N(v)\neq \emptyset \neq V_3 \cap N(v)$. Similarly, since $V_4 V_5$ and $V_5 V_6$ are also edges of $H$, therefore $V_4 \cap N(v)\neq \emptyset \neq V_6 \cap N(v)$. But this is a contradiction since the degree of $v$ cannot be greater than 3.

    If $V_1 V_2 V_3 V_4 V_5 V_6$ is a cycle of length 6, then the same argument\footnote{by rotating the path cyclically} shows that $V_1 V_4$, $V_2 V_5$ and $V_3 V_6$ are all mandatory in $H$.
    However, the maximum degree is at most 3 in $H$ by Lemma \ref{l:maxfok}, thus $H\simeq H_{13}$.
\end{proof}

By Lemma \ref{lem:6kor}, $H_4$, $H_5$, $H_9$, $H_{11}$ and $H_{12}$ are not realisable since they contain either a path of length 5 or a cycle of length 6 but some of the mandatory edges are missing. The rest of the possible optimal total coalition graphs can be realised, see Figures \ref{fig:maradek1}, \ref{fig:maradek3}, \ref{fig:maradek2}.

\begin{figure}[!h]
    \centering
    \includegraphics[width=0.65\textwidth]{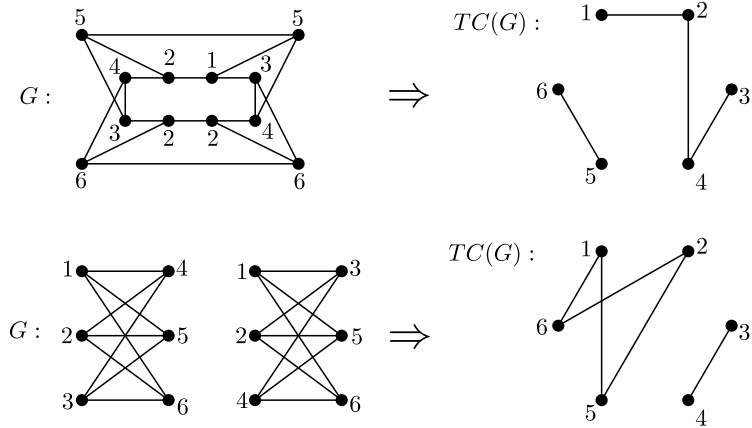}
    \caption{$H_2$ and $H_3$ are realisable.}
    \label{fig:maradek1}
\end{figure}

\begin{figure}[!h]
    \centering
    \includegraphics[width=0.55\textwidth]{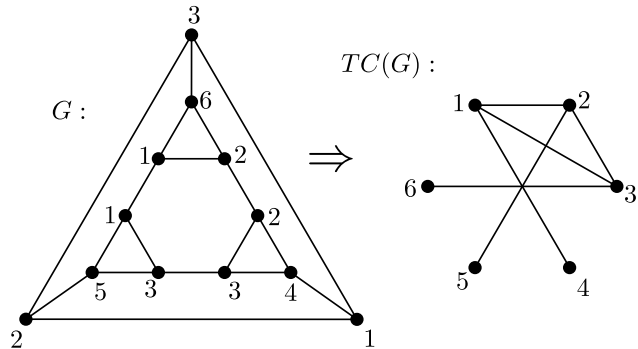}
    \caption{$H_7$ is also realisable.}
    \label{fig:maradek3}
\end{figure}

\begin{figure}[!h]
    \centering
    \includegraphics[width=0.9\textwidth]{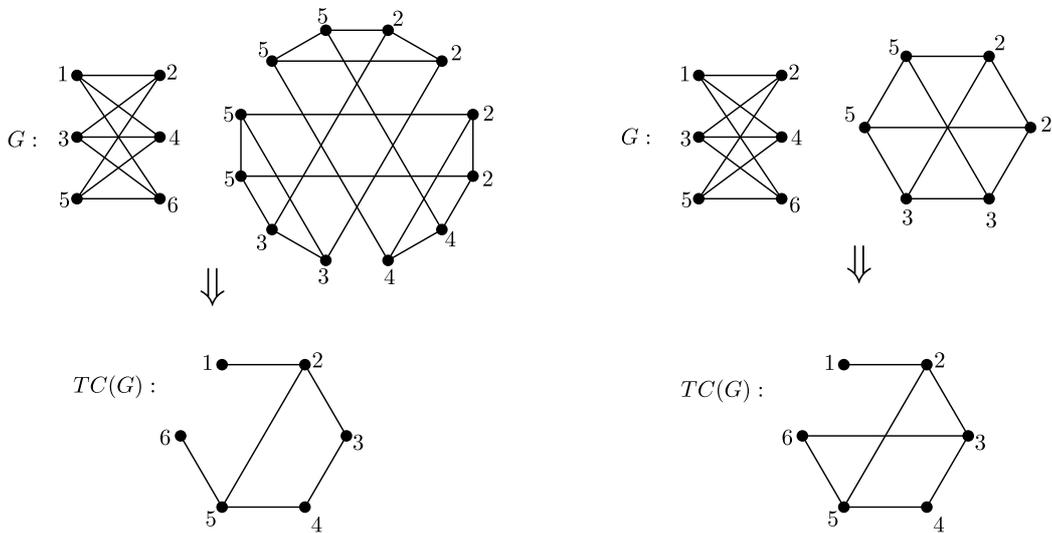}
    \caption{$H_8$ and $H_{10}$ are realisable, as well.}
    \label{fig:maradek2}
\end{figure}

\eject

\end{document}